\documentclass[review]{elsarticle}
\usepackage{exscale}
\usepackage{multirow}
\usepackage{multicol}
\usepackage{relsize}
\usepackage{lineno,hyperref}
\modulolinenumbers[5]

\usepackage{amssymb}
\newcommand{\upcite}[1]{$^{\mbox{\scriptsize \cite{#1}}}$}

\usepackage{geometry}
\usepackage{graphicx}
\usepackage{bm}
\usepackage{color,xcolor}
\usepackage{booktabs}
\usepackage{amsmath}
\usepackage{amsthm}

\usepackage{amsfonts}
\usepackage{subfigure}
\usepackage{float}
\usepackage[ruled]{algorithm2e}

\newtheorem{lemma}{Lemma}[section]
\newtheorem{theorem}{Theorem}[section]

\begin{document}

\begin{frontmatter}

\title{Convergence analysis of Jacobi spectral collocation methods for weakly singular nonlocal diffusion equations with volume constraints}

\author[mymainaddress]{Jiashu Lu}
\author[mymainaddress]{Mengna Yang}

\author[mymainaddress]{Yufeng Nie\corref{mycorrespondingauthor}}
\cortext[mycorrespondingauthor]{Corresponding author}
\ead{yfnie@nwpu.edu.cn}

\address[mymainaddress]{Xi'an Key Laboratory of Scientific Computation and Applied Statistics,
Northwestern Polytechnical University,
Xi'an, Shaanxi 710129, China. }

\begin{abstract}
This paper considers efficient spectral solutions for weakly singular nonlocal diffusion equations with Dirichlet-type volume constraints. The equation we consider contains an integral operator that typically has a singularity at the midpoint of the integral domain, and the approximation of the integral operator is one of the essential difficulties in solving nonlocal equations. To overcome this problem, two-sided Jacobi spectral quadrature rules are proposed to develop a Jacobi spectral collocation method for nonlocal diffusion equations. A rigorous convergence analysis of the proposed method with the $L^\infty$ norm is presented, and we further prove that the Jacobi collocation solution converges to its corresponding local limit as nonlocal interactions vanish. Numerical examples are given to verify the theoretical results.
\end{abstract}

\begin{keyword}
nonlocal diffusion equations \sep spectral collocation methods \sep weakly singular kernel \sep spectral accuracy \sep Jacobi quadrature
\end{keyword}

\end{frontmatter}

\section{Introduction}
Nonlocal models with volume constraints have received much attention in the last two decades due to their capacity to represent physical phenomena that cannot be effectively characterized by partial differential equation (PDE) models. Unlike local PDE models, nonlocal models replace the usual spatial differential operators with integral operators. Thus, the usual continuity and smoothness requirements for displacements are eliminated, which makes it possible to describe discontinuities, such as crack propagation\upcite{bobaru2015cracks} and the general Markov jump processes in bounded domains\upcite{du2014nonlocal}.

In this paper, we consider weakly singular nonlocal diffusion equations with volume constraints\upcite{2012Analysis}. Let $\Omega \subset \mathbb R$ be a bounded, open domain in $\mathbb R$. By defining $\delta$ as the scope of the nonlocal interaction, the corresponding interaction domain is then defined as
\begin{equation*}
\Omega_c=\{x \in \mathbb R\backslash\Omega:\text{dist}(x,\partial\Omega)\le\delta\}.
\end{equation*}
Then, nonlocal diffusion equations with Dirichlet-type volume constraints can be formulated as follows:
\begin{equation}\label{modelequation}
	\left\{\begin{aligned}
		&-\mathcal{L}_\delta u(x)=f(x),&x\in\Omega,\;\\
		&u(x)=g(x),&x\in\Omega_c,
		\end{aligned}\right.
\end{equation}
where $f$ and $g$ are the prescribed functions and $\mathcal{L}_\delta $ denotes the weakly singular nonlocal diffusion operator, which is defined as
\begin{equation}\label{nonlocaloperator}
	\mathcal{L}_\delta u(x)=\int_{x-\delta}^{x+\delta}\gamma_\delta(x,y)\frac{1}{|y-x|^\mu}\left(u(y)-u(x)\right)dy,
\end{equation}
where $0<\mu<1$ and $\gamma_\delta$ denote nonnegative and symmetric radial-type kernel functions.

Although nonlocal models can provide better modeling capabilities than traditional PDE models, they often lead to more computational difficulties when usual discretization methods are used. Such difficulties mainly come from the following two parts: the approximation of a singular nonlocal integral, and the numerical solution that yielded from denser discrete equations compared with the solutions to discrete PDE models. To date, many works have been performed to overcome one of the above difficulties. For example, a fast convolution-based method\upcite{jafarzadeh2021fast} reduces the computational cost of a singular nonlocal integral, a fast Fourier transform (FFT)-type method\upcite{2013AN,WANG201419} achieves high computing efficiency in the process of solving discrete equations, a localized radial bias function (RBF) collocation method significantly reduces the condition number of a nonlocal stiffness matrix\upcite{lu2021collocation}, and a proper orthogonal decomposition (POD)-based fast algorithm\upcite{shangyuan2020} accelerates the computation of dense discrete equations in time iterations. However, limited work has been done to overcome both of the above two difficulties, and we fill this gap by developing a Jacobi spectral collocation method in this paper.

Spectral methods have now become one of the most popularly used methods for the discretization of spatial variables in solving PDEs\upcite{shen2011spectral,jia2020numerical} since they can provide very accurate approximations for sufficiently smooth solutions, and they have been successfully used to solve integral equations, e.g., Volterra integral equations\upcite{xie2012convergence,samadi2012spectral,yao2022hybrid}, Fredholm integral equations\upcite{panigrahi2019legendre,benyoussef2020efficient} and fractional differential equations\upcite{wang2018spectral,gu2021spectral,guo2021linearized}. In \cite{tian2020spectral}, Tian et al. developed a Legendre spectral method for nonlocal diffusion equations that can achieve a fixed accuracy with fewer unknowns; therefore, it alleviates the latter part of the above computational difficulties by significantly reducing the degrees of freedom of the discrete equations. In light of the above work, we develop a Jacobi spectral collocation method for weakly singular nonlocal diffusion equations with rigorous convergence analyses. Inspired by \cite{yang2021using}, two-sided Jacobi spectral quadrature rules are proposed to accurately approximate singular nonlocal integrals. Combined with the above ideas, we completely overcome the two computational difficulties brought by the nonlocality of a nonlocal operator. Furthermore, we provide relevant theoretical analyses, which show that the numerical solution of a nonlocal diffusion equation converges to the correct local limit as $\delta\rightarrow 0$ and $N\rightarrow\infty$. To verify the theoretical result, numerical experiments are also presented for nonlocal diffusion equations with volume constraints.

The layout of the paper is as follows. In Section 2, the Jacobi spectral collocation methods are presented for weakly singular nonlocal diffusion equations. Some useful lemmas are provided in Section 3 for convergence and asymptotic compatibility analysis. An error analysis using the $L^\infty$ norm is given in Section 4, and we present an asymptotic compatibility analysis using the $L^2$ norm in Section 5. Numerical experiments are carried out in Section 6 to justify the theoretical results and demonstrate the efficiency of the proposed methods. In Section 7, we close the paper with some conclusions.

\section{Jacobi-collocation methods}
Without loss of generality, we choose $\Omega=I:=(-1,1)$ for convenience, and we note that one can treat more general domains through a simple linear transformation. Here, $\omega^{\alpha,\beta}(x)=(1-x)^\alpha(1+x)^\beta$ is the Jacobi weight function with $\alpha, \beta>-1$ and $x\in I $.\\
Let $L^2_{\omega^{\alpha,\beta}}(I)$ be a weighted Hilbert space
\begin{equation*}
	L^2_{\omega^{\alpha,\beta}}(I)=\left\{v:v\text{ is measurable and } ||v||_{\omega^{\alpha,\beta}}<\infty\right\},
\end{equation*}
equipped with the norm
\begin{equation*}
	||v||_{\omega^{\alpha,\beta}}=\left(\int_{-1}^1|v(x)|^2\omega^{\alpha,\beta}(x)dx\right)^{\frac{1}{2}},
\end{equation*}
and the weighted inner product
\begin{equation}\label{innerproduct}
	(u,v)_{\omega^{\alpha,\beta}}=\int_{-1}^1u(x)v(x)\omega^{\alpha,\beta}(x)dx.
\end{equation}

Before presenting the Jacobi collocation scheme for the nonlocal diffusion models, we rewrite the model equation in \eqref{modelequation} as
\begin{equation}\label{modelequation2}
\left\{\begin{aligned}
	&C_\delta u(x)-\int_{x-\delta}^{x+\delta}\gamma_\delta(x,y)\frac{1}{|y-x|^\mu}u(y)dy=f(x),&x\in I,\;\\
	&u(x)=g(x),&x\in I_c.
\end{aligned}\right.
\end{equation}
where
\begin{equation}\label{cdelta}
	C_\delta=\int_{x-\delta}^{x+\delta}\gamma_\delta(x,y)\frac{1}{|y-x|^\mu}dy=\int_{-\delta}^{\delta}\gamma_\delta(0,y)\frac{1}{|y|^\mu}dy
\end{equation}

To approximate the nonlocal integral in \eqref{modelequation2}, we introduce two-sided Jacobi spectral quadrature rules. First, by denoting
\begin{equation}\label{operators}
	S_1u(x)=\int_{x-\delta}^x\gamma_\delta(x,s_1)\frac{1}{(x-s_1)^\mu}u(s_1)ds_1,\quad S_2u(x)=\int_x^{x+\delta}\gamma_\delta(x,s_2)\frac{1}{(s_2-x)^\mu}u(s_2)ds_2,\quad x\in I,
\end{equation}
and employing the following change of variables:
\begin{equation}\label{trans}
	s_1(x,\theta)=\frac{2x-\delta}{2}+\frac{\delta}{2}\theta,\quad s_2(x,\theta)=\frac{2x+\delta}{2}+\frac{\delta}{2}\theta,\quad\theta\in[-1,1],
\end{equation}
one can rewrite \eqref{modelequation2} as
\begin{equation}\label{modelequation3}\small
\begin{aligned}
	&C_\delta u(x)-S_1u(x)-S_2u(x)\\
	&=C_\delta u(x)-\left(\frac{\delta}{2}\right)^{1-\mu}\left(\int_{-1}^{1}(1-\theta)^{-\mu}\gamma_\delta(x,s_1(x,\theta))u(s_1(x,\theta))d\theta+\int_{-1}^{1}(1+\theta)^{-\mu}\gamma_\delta(x,s_2(x,\theta))u(s_2(x,\theta))d\theta\right)\\
	&=f(x).
\end{aligned}
\end{equation}

For a given positive integer $N$, we denote the collocation points by $\{x_i\}_{i=1}^{N+1}$, which is the set of $N+1$ Jacobi-Gauss-Lobatto points\upcite{shen2011 spectral} corresponding to the weight function $\omega^{\alpha,\beta}(x)$, and we let $\mathbb P_N$ be the space of all polynomials with degrees not exceeding $N$. Then, the two-sided Jacobi spectral quadrature rules are given to approximate $S_1(x_i)$ and $S_2(x_i)$ using Jacobi spectral quadratures with different weight functions.

Let $\{\theta^1_j,\omega^1_j\}_{j=1}^M$ and $\{\theta^2_j,\omega^2_j\}_{j=1}^M$ be the sets of Jacobi-Gauss points and weights with weight functions $\omega^{-\mu,0}(x)$ and $\omega^{0,-\mu}$, respectively. For $\xi=1,2$, we set
\begin{equation*}
\begin{aligned}
&H^\xi_\text{in}(x_i)=\{j=1,...,M:s_\xi(x_i,\theta^\xi_j)\in I\},\\
&H^\xi_\text{out}(x_i)=\{j=1,...,M:s_\xi(x_i,\theta^\xi_j)\in I_c\},
\end{aligned}
\end{equation*}
Then, the Jacobi collocation scheme for \eqref{modelequation3} becomes a process of finding the approximate solution $u_N\in\mathbb P_N$ such that
\begin{equation}\label{collocationscheme}\small
\begin{aligned}
&C_\delta u_N(x_i)-\left(\frac{\delta}{2}\right)^{1-\mu}\left(\sum_{j\in H^1_\text{in}(x_i)}\omega_j^1\gamma_\delta(x_i,s_1(x_i,\theta_j^1))u_N(s_1(x_i,\theta_j^1))+\sum_{j\in H^2_\text{in}(x_i)}\omega_j^2\gamma_\delta(x_i,s_2(x_i,\theta_j^2))u_N(s_2(x_i,\theta_j^2))\right)\\
&=f(x_i)+\left(\frac{\delta}{2}\right)^{1-\mu}\left(\sum_{j\in H^1_\text{out}(x_i)}\omega_j^1\gamma_\delta(x_i,s_1(x_i,\theta_j^1))g(s_1(x_i,\theta_j^1))+\sum_{j\in H^2_\text{out}(x_i)}\omega_j^2\gamma_\delta(x_i,s_2(x_i,\theta_j^2))g(s_2(x_i,\theta_j^2))\right).
\end{aligned}	
\end{equation}
In this paper, the case $M=N$ is selected for convenience of analysis, and one can treat the other cases similarly. Then, the Lagrange interpolation polynomial $I_N^{\alpha,\beta}u\in\mathbb P_N$ corresponding to function $u$ is defined as
\begin{equation*}
	I_N^{\alpha,\beta}u(x_i)=u(x_i),\quad 0\le i\le N.
\end{equation*}
We seek an approximate solution $u_N(x)$ of the form
\begin{equation}\label{interpolation}
	u_N(x)=I^{\alpha,\beta}_Nu(x)=\sum_{k=1}^{N+1}u_kh_k(x),
\end{equation}
where $\{h_k(x)\}_{k=1}^{N+1}$ is the Lagrange interpolation basis function.

Substituting \eqref{interpolation} into \eqref{collocationscheme} leads to
\begin{equation}\label{fulldiscrete}\small
\begin{aligned}
&C_\delta u_i-\left(\frac{\delta}{2}\right)^{1-\mu}\sum_{k=1}^{N+1}u_k\left(\sum_{j\in H^1_\text{in}(x_i)}\omega_j^1\gamma_\delta(x_i,s_1(x_i,\theta_j^1))h_k(s_1(x_i,\theta_j^1))+\sum_{j\in H^2_\text{in}(x_i)}\omega_j^2\gamma_\delta(x_i,s_2(x_i,\theta_j^2))h_k(s_2(x_i,\theta_j^2))\right)\\
&=f(x_i)+\left(\frac{\delta}{2}\right)^{1-\mu}\left(\sum_{j\in H^1_\text{out}(x_i)}\omega_j^1\gamma_\delta(x_i,s_1(x_i,\theta_j^1))g(s_1(x_i,\theta_j^1))+\sum_{j\in H^2_\text{out}(x_i)}\omega_j^2\gamma_\delta(x_i,s_2(x_i,\theta_j^2))g(s_2(x_i,\theta_j^2))\right).
\end{aligned}	
\end{equation}
By denoting $\mathbf u_N=(u_1,u_2,...,u_{N+1})$, we can rewrite \eqref{fulldiscrete} in the following matrix form:
\begin{equation}\label{discreteequation}
	\mathbf A \mathbf u_N = \mathbf F,
\end{equation}
where the entries of the matrix $\mathbf A$ can be computed as
\begin{equation*}
	(\mathbf A)_{i,k}=C_\delta\delta_{ki}-\left(\frac{\delta}{2}\right)^{1-\mu}\left(\sum_{j\in H^1_\text{in}(x_i)}\omega_j^1\gamma_\delta(x_i,s_1(x_i,\theta_j^1))h_k(s_1(x_i,\theta_j^1))+\sum_{j\in H^2_\text{in}(x_i)}\omega_j^2\gamma_\delta(x_i,s_2(x_i,\theta_j^2))h_k(s_2(x_i,\theta_j^2))\right),
\end{equation*}
where $\delta_{ki}$ is the Kronecker-Delta symbol, and the entries of the source vector $\mathbf F$ are
\begin{equation*}\small
	(\mathbf F)_i=f(x_i)+\left(\frac{\delta}{2}\right)^{1-\mu}\left(\sum_{j\in H^1_\text{out}(x_i)}\omega_j^1\gamma_\delta(x_i,s_1(x_i,\theta_j^1))g(s_1(x_i,\theta_j^1))+\sum_{j\in H^2_\text{out}(x_i)}\omega_j^2\gamma_\delta(x_i,s_2(x_i,\theta_j^2))g(s_2(x_i,\theta_j^2))\right).
\end{equation*}
Hence, one can obtain the coefficient vector $\mathbf u_N$ by solving \eqref{discreteequation}. The solvability theorem of the linear system in \eqref{fulldiscrete} is given in Section \ref{convergencesection}.
\section{Some useful lemmas}
To obtain the convergence analysis, we introduce some useful lemmas in this section. First, for a nonnegative integer $m$, we define
\begin{equation*}
	H^m_{\omega^{\alpha,\beta}}(I):=\{v:\partial^k_x v\in L^2_{\omega^{\alpha,\beta}}(I),0\le k\le m\},
\end{equation*}
which is equipped with the norm
\begin{equation*}
	||v||_{m,\omega^{\alpha,\beta}}=\left(\sum_{k=1}^{m}||\partial^k_x v||^2_{\omega^{\alpha,\beta}}\right)^{\frac{1}{2}}.
\end{equation*}
Then, it is convenient to introduce the following seminorm:
\begin{equation*}
	|v|_{\omega^{\alpha,\beta}}^{m;N}:=\left(\sum_{k=\min(m,N+1)}^m||\partial^k_xv||_{\omega^{\alpha,\beta}}^2\right)^{\frac{1}{2}}.
\end{equation*}
We further denote $H^r(I):=H^r_{\omega^{0,0}}(I)$ as a regular Hilbert space equipped with the norm $||v||_{m}$ and the seminorm $|v|^{m;N}$. Moreover, for any $u,v\in C[-1,1]$, we define a discrete inner product $\left<u,v\right>_{\omega^{\alpha,\beta},N}$ as
\begin{equation*}
	\left<u,v\right>_{N,\omega^{\alpha,\beta}}=\sum_{i=1}^{N+1}u(\theta_i)v(\theta_i)\omega_i,
\end{equation*}
where $\{\theta_i,\omega_i\}_{i=1}^{N+1}$ represents the Gauss points and Gauss weights corresponding to the weight functions $\omega^{\alpha,\beta}(x)$.
\begin{lemma}{\cite[Lemma 3.1, 3.2]{chen2013note}} If $v \in H^m_{\omega^{\alpha,\beta}}(I)$ for some $m\ge1$, $-1<\alpha,\beta<1$, and $\phi\in \mathbb P_N$, then the following estimate holds:\label{quadratureerror}\label{interpolationerror2}
\begin{equation*}
	\begin{aligned}
	\left|(v,\phi)_{\omega^{\alpha,\beta}}-\left<v,\phi\right>_{N,\omega^{\alpha,\beta}}\right|&\le c N^{-m} |v|_{\omega^{\alpha,\beta}}^{m;N}||\phi||_{\omega^{\alpha,\beta}},\\
	||v-I_N^{\alpha,\beta}v||_{\omega^{\alpha,\beta}}&\le c N^{-m}|v|_{\omega^{\alpha,\beta}}^{m;N},\\
	||v-I_N^{\alpha,\beta}v||_{1,\omega^{\alpha,\beta}}&\le c N^{1-m}|v|_{\omega^{\alpha,\beta}}^{m;N}.
	\end{aligned}
	\end{equation*}
\end{lemma}
Additionally, \cite{sohrabi2017convergence} gives the following lemma to estimate the $L^\infty$ norm of the interpolation operator $I^{\alpha,\beta}_N$:
\begin{lemma}{\cite[Lemma 1]{sohrabi2017convergence}}\label{interpolationerror}
Let $\{h_k(x)\}_{k=1}^{N+1}$ be the Lagrange interpolation basis functions associated with the Jacob-Gauss points corresponding to the weight function $\omega^{\alpha,\beta}(x)$ with $\alpha,\beta>-1$, and denote $\gamma=\max(\alpha,\beta)$. Then,
\begin{equation*}
			||I^{\alpha,\beta}_N||_\infty:=\max_{x\in I}\sum_{k=1}^{N+1}|h_k(x)|=\left\{
			\begin{aligned}
			&O(\ln N),&-1<\alpha,\beta\le-\frac{1}{2},\\
			&O(N^{\gamma+\frac{1}{2}}),&\textup{otherwise}.
	\end{aligned}\right.
	\end{equation*}
\end{lemma}
For an integer $r\ge0$ and $0\le\kappa\le1$, let $C^{r,\kappa}(I)$ be the space of functions whose $r$-th derivatives are H{\"o}lder continuous with exponent $\kappa$, and this space is equipped with the usual norm
\begin{equation*}
	||v||_{C^{r,\kappa}}=\max_{0\le l\le r}\max_{x\in I}|\partial^l_x v(x)|+\max_{0\le l\le r}\sup_{x\neq y}\frac{|\partial ^l_x v(x)-\partial^l_x v(y)|}{|x-y|^\kappa}.
\end{equation*}
Then, \cite{shen2011spectral} provided the following lemma:
\begin{lemma}{\cite[Lemma 5.1]{shen2011spectral}}\label{holderestimate}
For any nonnegative integer $r$ and $\kappa\in(0,1)$, there exists a linear transform $T_N:C^{r,\kappa}\rightarrow \mathbb P_N$ and a positive constant $c_{r,\kappa}$ such that
\begin{equation*}
		||v-T_Nv||_{L^\infty(I)}\le c_{r,\kappa}N^{-(r+\kappa)}||v||_{C^{r,\kappa}}.
	\end{equation*}
\end{lemma}

\begin{lemma}\label{lemma2}
Let $x_1<x_2$ and $0<\mu<1$; the following inequality holds:
\begin{subequations}
\begin{align}
\int_{x_1-\delta}^{x_1}\left[(x_1-\tau)^{-\mu}-(x_2-\tau)^{-\mu}\right]d\tau\le c|x_2-x_1|^{1-\mu},\label{lemma3a}\\
\int_{x_2}^{x_2+\delta}\left[(\tau-x_2)^{-\mu}-(\tau-x_1)^{-\mu}\right]d\tau\le c|x_2-x_1|^{1-\mu},\label{lemma3b}
\end{align}
\end{subequations}
where $c$ is a constant that depends on $\mu$.
\end{lemma}
\begin{proof}
Since $x_1<x_2$, we have
\begin{equation*}
\left[(x_1-\tau)^{-\mu}-(x_2-\tau)^{-\mu}\right]>0,\quad\tau\in[x_1-\delta,x_1],
\end{equation*}
which leads to
\begin{equation}\label{estimate1}
\begin{aligned}
&\int_{x_1-\delta}^{x_1}\left[(x_1-\tau)^{-\mu}-(x_2-\tau)^{-\mu}\right]d\tau\\
&\le \int_{-1-\delta}^{x_1}\left[(x_1-\tau)^{-\mu}-(x_2-\tau)^{-\mu}\right]d\tau\\
&\le \left|\int_{-1-\delta}^{x_1}(x_1-\tau)^{-\mu}d\tau-\int_{-1-\delta}^{x_2}(x_2-\tau)^{-\mu}d\tau\right|+\left|\int_{x_1}^{x_2}(x_2-\tau)^{-\mu}d\tau\right|\\
&\le \left[\left(\frac{x_2+1+\delta}{2}\right)^{1-\mu}-\left(\frac{x_1+1+\delta}{2}\right)^{1-\mu}\right]\int_{-1}^{1}(1-\theta)^{-\mu}d\theta+\frac{|x_2-x_1|^{1-\mu}}{1-\mu},
\end{aligned}
\end{equation}
where the last inequality follows from the linear transformation
\begin{equation*}
	\tau=\frac{x_\xi+1+\delta}{2}\theta+\frac{x_\xi-1-\delta}{2},\quad \xi=1,2.
\end{equation*}
One can easily observe that
\begin{equation}\label{estimate2}
\begin{aligned}
	&\left(\frac{x_2+1+\delta}{2}\right)^{1-\mu}-\left(\frac{x_1+1+\delta}{2}\right)^{1-\mu}=\frac{1-\mu}{2^{1-\mu}}\int_{x_1}^{x_2}(y+1+\delta)^{-\mu}dy\\
	&\le\frac{1-\mu}{2^{1-\mu}}\int_{x_1}^{x_2}(y-x_1)^{-\mu}dy=2^{\mu-1}|x_2-x_1|^{1-\mu},
\end{aligned}
\end{equation}
where we use the fact that $y+1+\delta>y-x_1$. Combining \eqref{estimate1} and \eqref{estimate2} gives the desired estimate in \eqref{lemma3a}. One can prove the estimate in \eqref{lemma3b} with a similar approach.
\end{proof}
Next, we show that the operators $S_1$ and $S_2$ defined in \eqref{operators} are bounded operators from $L^\infty(I\cup I_c)$ to $C^{0,\kappa}(I)$ with $0<\kappa<1-\mu$, which plays a key role in the error analysis in the next section.
\begin{lemma}\label{nonlocalstibility}
Let $0<\mu<1$ and $0<\kappa<1-\mu$; then, for any function $v\in C(I\cup I_c)$ and any $-1<x_1,x_2<1$ with $x_1\neq x_2$, there exists a positive constant $c$ such that
\begin{equation*}
		\frac{|S_\xi v(x_1)-S_\xi v(x_2)|}{|x_1-x_2|^\kappa}\le c||v||_{L^\infty(I\cup I_c)},
	\end{equation*}
where $\xi=1,2$, which implies
\begin{equation*}
		||S_\xi v||_{C^{0,\kappa}}\le c||v||_{L^\infty(I\cup I_c)}.
	\end{equation*}
where $c$ depends on $||\gamma_\delta||_{C^{0,1-\mu}}$ and $||\gamma_\delta||_{L^\infty(D)}$ with $D=\left(I\cup I_c\right)\times\left(I\cup I_c\right)$.
\end{lemma}
\begin{proof}
Without loss of generality, we assume that $x_1<x_2$. Then, for $\xi=1$, the use of the triangle inequality leads to
\begin{equation*}
	\begin{aligned}
		&|S_1v(x_1)-S_1v(x_2)|\\
		&=\left|\int_{x_1-\delta}^{x_1}\gamma_\delta(x_1,s)\frac{1}{(x_1-s)^\mu}v(s)ds-\int_{x_2-\delta}^{x_2}\gamma_\delta(x_2,s)\frac{1}{(x_2-s)^\mu}v(s)ds\right|\\
		&\le \left|\int_{x_1-\delta}^{x_1}\gamma_\delta(x_1,s)\frac{1}{(x_1-s)^\mu}v(s)ds-\int_{x_1-\delta}^{x_1}\gamma_\delta(x_2,s)\frac{1}{(x_2-s)^\mu}v(s)ds\right|\\
		&+\left|\left(\int_{x_1-\delta}^{x_2-\delta}-\int_{x_1}^{x_2}\right)\gamma_\delta(x_2,s)\frac{1}{(x_2-s)^\mu}v(s)ds\right|\\
		&\le \int_{x_1-\delta}^{x_1}\left|(x_1-s)^{-\mu}-(x_2-s)^{-\mu}\right|\left|\gamma_\delta(x_1,s)\right|\left|v(s)\right|ds\\
		&+\int_{x_1-\delta}^{x_1}(x_2-s)^{-\mu}\left|\gamma_\delta(x_1,s)-\gamma_\delta(x_2,s)\right|\left|v(s)\right|ds\\
		&+\int_{x_1-\delta}^{x_2-\delta}(x_2-s)^{-\mu}\left|\gamma_\delta(x_2,s)\right|\left|v(s)\right|ds+\int_{x_1}^{x_2}(x_2-s)^{-\mu}\left|\gamma_\delta(x_2,s)\right|\left|v(s)\right|ds\\
		&:=E_1+E_2+E_3+E_4.
	\end{aligned}
	\end{equation*}
We now estimate the above four terms one by one. First, by Lemma \ref{lemma2},\\
\begin{equation*}
	\begin{aligned}
	    	E_1&\le ||v||_{L^\infty(I\cup I_c)}||\gamma_\delta||_{L^\infty(D)}\int_{x_1-\delta}^{x_1}\left|(x_1-s)^{-\mu}-(x_2-s)^{-\mu}\right|ds\\
	    	&\le c ||v||_{L^\infty(I\cup I_c)}|x_2-x_1|^{1-\mu},
	\end{aligned}
	\end{equation*}
where $c$ depends on $ ||\gamma_\delta||_{L^\infty(D)}$. Then, we estimate $E_2$
\begin{equation*}
		\begin{aligned}
		E_2&\le||v||_{L^\infty(I\cup I_c)}|x_2-x_1|^{1-\mu}\int_{x_1-\delta}^{x_1}(x_2-s)^{-\mu}\frac{\left|\gamma_\delta(x_1,s)-\gamma_\delta(x_2,s)\right|}{|x_2-x_1|^{1-\mu}}ds\\
		&\le ||v||_{L^\infty(I\cup I_c)}|x_2-x_1|^{1-\mu}||\gamma_\delta||_{C^{0,1-\mu}}\frac{1}{1-\mu}\left[\left(x_2-x_1+\delta\right)^{1-\mu}-\left(x_2-x_1\right)^{1-\mu}\right]\\
		&\le c||v||_{L^\infty(I\cup I_c)}|x_2-x_1|^{1-\mu},
		\end{aligned}
	\end{equation*}
where $c$ depends on $||\gamma_\delta||_{C^{0,1-\mu}}$. Moreover, we have
\begin{equation*}
		\begin{aligned}
		E_3&\le\delta^{-\mu}||\gamma_\delta||_{L^\infty(D)}||v||_{L^\infty(I\cup I_c)}|x_2-x_1|\\
		&\le c||v||_{L^\infty(I\cup I_c)}|x_2-x_1|^{1-\mu},
		\end{aligned}
	\end{equation*}
where the constant $c$ depends on $||\gamma_\delta||_{L^\infty(D)}$. Finally, we begin to estimate $E_4$
\begin{equation*}
		\begin{aligned}
		E_4&\le ||\gamma_\delta||_{L^\infty(D)}||v||_{L^\infty(I\cup I_c)}\int_{x_1}^{x_2}(x_2-s)^{-\mu}ds\\
		&\le c ||v||_{L^\infty(I\cup I_c)} |x_2-x_1|^{1-\mu},
		\end{aligned}
	\end{equation*}
where $c$ depends on $||\gamma_\delta||_{L^\infty(D)}$. The combination of the above estimates completes the proof of the lemma when $\xi=1$, and the case when $\xi=2$ can be handled in the same manner.
\end{proof}
\begin{lemma}{\cite[Lemma 4.5]{tian2020spectral}}\label{nonlocalestimate}
Let $p(t)$ and $Q(t)$ be integrable functions that satisfy
\begin{equation*}
		C_\delta p(t) -\int_{t-\delta}^{t+\delta}p(y)\gamma_\delta(t,y)\frac{1}{|y-t|^\mu}dy=Q(t),\quad t\in(-1,1),
	\end{equation*}
where $C_\delta$ is defined as \eqref{cdelta} and
\begin{equation*}
		p(t)=0,\quad t\in(-1-\delta,-1]\cup[1,1+\delta).
	\end{equation*}
Then,
\begin{equation*}
		||p(t)||_{L^\infty(I)}\le c||Q||_{L^\infty(I)}.
	\end{equation*}
\end{lemma}
\section{Convergence analysis}\label{convergencesection}
First, we present a solvability theorem for the proposed numerical method \eqref{fulldiscrete}.
\begin{theorem}
Assume $0<\mu<1$ and $-1<\alpha,\beta\le1$. Then, the collocation scheme in \eqref{fulldiscrete} leads to a unique numerical solution $u_N\in \mathbb P_N$, where  $N$ is sufficiently large.
\end{theorem}
\begin{proof}
We prove that the assumptions $f(x)=0, x\in I$ and $g(x)=0, x\in I_c$ lead to a unique solution $u_N=0$. First, we define
\begin{equation*}
	\bar{u}_N=\left\{
	\begin{aligned}
	&u_N(x)=\sum_{i=1}^{N+1}u_ih_i(x),&\forall x\in I,\;\\
	&g(x),&\forall x\in I_c.
	\end{aligned}\right.
	\end{equation*}
Thus, the discrete equation in \eqref{fulldiscrete} can be written as
\begin{equation*}
		 C_\delta \bar u_N(x_i)-\left(\frac{\delta}{2}\right)^{1-\mu}\left[\left<\gamma_\delta(x_i,s_1(x_i,\cdot)),\bar{u}_N(s_1(x_i,\cdot))\right>_{N,\omega^{-\mu,0}}+\left<\gamma_\delta(x_i,s_2(x_i,\cdot)),\bar u_N(s_2(x_i,\cdot))\right>_{N,\omega^{0,-\mu}}\right]=0,
	\end{equation*}
which is equivalent to
\begin{equation*}
		\begin{aligned}
		&C_\delta \bar u_N(x)-I^{\alpha,\beta}_N\left\{\left(\frac{\delta}{2}\right)^{1-\mu}\left[\left<\gamma_\delta(x,s_1(x,\cdot)),\bar{u}_N(s_1(x,\cdot))\right>_{N,\omega^{-\mu,0}}+\left<\gamma_\delta(x,s_2(x,\cdot)),\bar u_N(s_2(x,\cdot))\right>_{N,\omega^{0,-\mu}}\right]\right\}\\
		&=C_\delta\bar u_N(x)-S\bar u_N(x)-I_N^{\alpha,\beta}G(x)=0,
		\end{aligned}
	\end{equation*}
where
\begin{equation}\label{G}
	\begin{aligned}
	G(x)&=\left(\frac{\delta}{2}\right)^{1-\mu}\left\{\left<\gamma_\delta(x,s_1(x,\cdot)),\bar u_N(s_1(x,\cdot))\right>_{N,\omega^{-\mu,0}}-\left(\gamma_\delta(x,s_1(x,\cdot)),\bar u_N(s_1(x,\cdot))\right)_{\omega^{-\mu,0}}\right\}\\
	&+\left(\frac{\delta}{2}\right)^{1-\mu}\left\{\left<\gamma_\delta(x,s_2(x,\cdot)),\bar u_N(s_2(x,\cdot))\right>_{N,\omega^{0,-\mu}}-\left(\gamma_\delta(x,s_2(x,\cdot)),\bar u_N(s_2(x,\cdot))\right)_{\omega^{0,-\mu}}\right\}.
	\end{aligned}
	\end{equation}
Then, it is clear that
\begin{equation*}
		C_\delta \bar u_N(x)-S\bar u_N(x)=I_N^{\alpha,\beta}G(x).
	\end{equation*}
The use of Lemma \ref{nonlocalestimate} implies that
\begin{equation}\label{theorem1estimate}
		||\bar u_N||_{L^\infty(I)}\le c||I_N^{\alpha,\beta}G||_{L^\infty(I)}.
	\end{equation}
Directly from Lemma \ref{quadratureerror}, we can deduce that
\begin{equation*}\small
	\begin{aligned}
	&||I^{\alpha,\beta}_NG||_{L^\infty(I)}\\
	&\le\max_{1\le i\le N+1}|G(x_i)|\max_{x\in I}\sum_{i=1}^{N+1}|h_i(x)|\\
	&\le c||I^{\alpha,\beta}_N||_\infty N^{-m}\times \\
	&\max_{1\le i\le N+1}\left(|\gamma_\delta(x_i,s_1(x_i,\cdot))|_{\omega^{0,-\mu}}^{m;N}||\bar u_N(s_1(x_i,\cdot))||_{\omega^{0,-\mu}}+|\gamma_\delta(x_i,s_2(x_i,\cdot))|_{\omega^{0,-\mu}}^{m;N}||\bar u_N(s_2(x_i,\cdot))||_{\omega^{0,-\mu}}\right).
	\end{aligned}
	\end{equation*}
A direct computation leads to
\begin{equation*}
	||\bar u_N(s_1(x_i,\cdot))||_{\omega^{-\mu,0}}=\left(\frac{2}{\delta}\right)^{\frac{1-\mu}{2}}\left(\int_{x_i-\delta}^{x_i}|\bar u_N(s)|^2(x_i-s)^{-\mu}ds\right)^\frac{1}{2}\le c||\bar u_N||_{L^\infty(I)}.		
	\end{equation*}
Similarly,
\begin{equation*}
	||\bar u_N(s_2(x_i,\cdot))||_{\omega^{0,-\mu}}\le c||\bar u_N||_{L^\infty(I)}.	
	\end{equation*}
Hence, we have
\begin{equation}\label{ING11}
	||I_N^{\alpha,\beta }G||_{L^\infty(I)}\le c||I_N^{\alpha,\beta}||_\infty K^*_m N^{-m} ||u_N||_{L^\infty(I)},
	\end{equation}
where
\begin{equation*}
	K^*_m:=\max_{1\le i\le N+1}\left\{|\gamma_\delta(x_i,s_1(x_i,\cdot))|_{\omega^{0,-\mu}}^{m;N}+|\gamma_\delta(x_i,s_2(x_i,\cdot))|_{\omega^{-\mu,0}}^{m;N}\right\}.
	\end{equation*}
The combination of \eqref{theorem1estimate}, \eqref{ING11} and Lemma \ref{interpolationerror} implies that the assumptions $f(x)=0, x\in I$ and $g(x)=0, x\in I_c$ lead to $u_N=0$ when $N$ is sufficiently large. Hence, the spectral collocation solution $u_N$ exists and is unique, as $\mathbb P_N$ is finite dimensional.
\end{proof}
Then, we present a convergence analysis for the proposed numerical method in \eqref{fulldiscrete} with the $L^\infty$ norm.
\begin{theorem}\label{convergence}
Assume $0<\mu<1$ and $-1<\alpha,\beta\le1$. Let $u$ and $u_N$ be the solutions of the nonlocal diffusion equation in \eqref{modelequation} and the collocation scheme in \eqref{fulldiscrete}, respectively. If $u\in H^r(I)$ and $r\ge1$ is an integer, then for sufficiently large $N$, we have the following error estimate:
\begin{equation}\label{errorestimate}
	\begin{aligned}
		||u-u_N||_{L^\infty(I)}\le \left\{\begin{aligned}
		&c\ln N\left[K^*_mN^{-m}||u||_{L^\infty(I)}+N^{\frac{1}{2}-r}|u|^{r;N}\right] ,&-1<\alpha,\beta\le -\frac{1}{2},\\
		&c\left[K^*_mN^{\gamma+\frac{1}{2}-m}||u||_{L^\infty(I)}+N^{1+\gamma-r}|u|^{r;N}\right],&\textup{otherwise},
		\end{aligned}\right.
	\end{aligned}
	\end{equation}
where $\gamma = \max\{\alpha,\beta\}$ and $c$ is a positive constant independent of $N,r,u$.
\end{theorem}
\begin{proof}
First, we define
\begin{equation*}
	\bar{u}_N=\left\{
	\begin{aligned}
	&u_N(x)=\sum_{i=1}^{N+1}u_ih_i(x),&\forall x\in I,\;\\
	&g(x),&\forall x\in I_c.
	\end{aligned}\right.
	\end{equation*}
Inserting the collocation points $\{x_i\}_{i=1}^{N+1}$ into the nonlocal diffusion equation in \eqref{modelequation3} leads to
\begin{equation*}
	\begin{aligned}
			&C_\delta u(x_i)-f(x_i)\\
			&=\left(\frac{\delta}{2}\right)^{1-\mu}\left[\left(\gamma_\delta(x_i,s_1(x_i,\cdot)),u(s_1(x_i,\cdot))\right)_{\omega^{-\mu,0}}+\left(\gamma_\delta(x_i,s_2(x_i,\cdot)),u(s_2(x_i,\cdot))\right)_{\omega^{0,-\mu}}\right]
	\end{aligned}
	\end{equation*}
for $1\le i\le N+1$. Then, \eqref{collocationscheme} is rewritten as
\begin{equation*}
	\begin{aligned}
			&C_\delta \bar{u}_N(x_i)-f(x_i)\\
			&=\left(\frac{\delta}{2}\right)^{1-\mu}\left[\left<\gamma_\delta(x_i,s_1(x_i,\cdot)),\bar{u}_N(s_1(x_i,\cdot))\right>_{N,\omega^{-\mu,0}}+\left<\gamma_\delta(x_i,s_2(x_i,\cdot)),\bar u_N(s_2(x_i,\cdot))\right>_{N,\omega^{0,-\mu}}\right],
	\end{aligned}
	\end{equation*}
for $1\le i\le N+1$.

Denoting $e=u-u_N$ and $\bar e=u-\bar u_N$, we have the following error equation:
\begin{equation}\label{errorequation}
		\begin{aligned}
		C_\delta \bar{e}(x_i)=\left(\frac{\delta}{2}\right)^{1-\mu}\left[\left(\gamma_\delta(x_i,s_1(x_i,\cdot)),\bar e(s_1(x_i,\cdot))\right)_{\omega^{-\mu,0}}+\left(\gamma_\delta(x_i,s_2(x_i,\cdot)),\bar e(s_2(x_i,\cdot))\right)_{\omega^{0,-\mu}}\right]-G(x_i),
		\end{aligned}
	\end{equation}
where $G(x)$ is defined in \eqref{G}. Then, we rewrite \eqref{errorequation} as
\begin{equation*}
		\begin{aligned}
		&C_\delta\left(I^{\alpha,\beta}_N u(x)-\bar u_N(x)\right)\\
		&=I_N^{\alpha,\beta}\left(\int_{x-\delta}^{x}(x-s)^{-\mu}\gamma_\delta(x,s)\bar e(s)ds+\int_x^{x+\delta}(s-x)^{-\mu}\gamma_\delta(x,s)\bar e(s)ds\right)-I_N^{\alpha,\beta}G(x).
		\end{aligned}
	\end{equation*}
Consequently,
\begin{equation*}
		\begin{aligned}
		C_\delta\bar{e}&=C_\delta(u-\bar u_N)\\
		&=C_\delta\left(u-I_N^{\alpha,\beta}u+I_N^{\alpha,\beta}u-\bar u_N\right)\\
		&=\int_{x-\delta}^{x+\delta}|s-x|^{-\mu}\left(\bar{e}(s)-\bar{e}(x)\right)ds\\
		&+G_1+G_2-I^{\alpha,\beta}_NG,
		\end{aligned}
	\end{equation*}
where
\begin{equation*}
		\begin{aligned}
		G_1&=C_\delta (u-I^{\alpha,\beta}_Nu)\\
		G_2&=I_N^{\alpha,\beta}\left(\int_{x-\delta}^{x}(x-s)^{-\mu}\gamma_\delta(x,s)\bar e(s)ds\right)-\int_{x-\delta}^{x}(x-s)^{-\mu}\gamma_\delta(x,s)\bar e(s)ds\\
		&+I_N^{\alpha,\beta}\left(\int_{x}^{x+\delta}(s-x)^{-\mu}\gamma_\delta(x,s)\bar e(s)ds\right)-\int_{x}^{x+\delta}(s-x)^{-\mu}\gamma_\delta(x,s)\bar e(s)ds.
		\end{aligned}
	\end{equation*}
It is worthwhile to note that $\bar u_N(x) = u(x)$ for $x \in I_c$, which leads to $\bar e(x)=0 $ for $x\in I_c$. Thus, by Lemma \ref{nonlocalestimate},
\begin{equation}\label{error1}
		||\bar{e}||_{L^\infty(I)}=||e||_{L^\infty(I)}\le c\left(||G_1||_{L^\infty(I)}+||G_2||_{L^\infty(I)}+||I^{\alpha,\beta}_NG||_{L^\infty(I)}\right).
	\end{equation}
From \eqref{ING11}, we have
\begin{equation}\label{ING1}
		||I_N^{\alpha,\beta }G||_{L^\infty(I)}\le c||I_N^{\alpha,\beta}||_\infty K^*_m N^{-m} (||u||_{L^\infty(I)}+||\bar e||_{L^\infty(I)}).
	\end{equation}
Combining estimate \eqref{ING1} and Lemma \ref{interpolationerror}, for sufficiently large $N$, we can obtain that
\begin{equation}\label{ING}
		||I_N^{\alpha,\beta }G||_{L^\infty(I)}\le\left\{\begin{aligned}
		&\frac{1}{3}||\bar e||_{L^\infty(I)}+cK^*_m N^{-m}\ln N||u||_{L^\infty(I)},&-1<\alpha,\beta\le\-\frac{1}{2},\\
		&\frac{1}{3}||\bar e||_{L^\infty(I)}+cK^*_m N^{\gamma+\frac{1}{2}-m}||u||_{L^\infty(I)},&\textup{otherwise}.
		\end{aligned}\right.
	\end{equation}
Furthermore, using Lemma \ref{interpolationerror2} and the Sobolev inequality\upcite{shen2011 spectral} gives
\begin{equation*}
	\begin{aligned}
		||G_1||_{L^\infty(I)}&=C_\delta ||u-I^{\alpha,\beta}_Nu||_{L^\infty(I)}=C_\delta ||u-I_Nu+I_N^{\alpha,\beta}(I_Nu-u)||_{L^\infty(I)}\\
		&\le C_\delta (1+||I^{\alpha,\beta}_N||_\infty)||u-I_Nu||_{L^{\infty}(I)}\\
		&\le C_\delta (1+||I^{\alpha,\beta}_N||_\infty)||u-I_Nu||_{\omega^{0,0}}^{\frac{1}{2}}||u-I_Nu||_{1,\omega^{0,0}}^{\frac{1}{2}}\\
		&\le c(1+||I^{\alpha,\beta}_N||_\infty)N^{\frac{1}{2}-r}|u|_{\omega^{0,0}}^{r;N},
	\end{aligned}
	\end{equation*}
where $I_N$ denotes the Legendre polynomial interpolation operator. Therefore, the combination of the above estimate and Lemma \ref{interpolationerror} implies
\begin{equation}\label{G1}
		||G_1||_{L^\infty(I)}\le
		\left\{\begin{aligned}
		&c N^{\frac{1}{2}-r}\ln N|u|_{\omega^{0,0}}^{r;N},&-1<\alpha,\beta\le-\frac{1}{2},\\
		&cN^{1+\gamma-r}|u|_{\omega^{0,0}}^{r;N},&\textup{otherwise}.
		\end{aligned}\right.
	\end{equation}
Finally, from Lemma \ref{holderestimate} and Lemma \ref{nonlocalstibility}, we have
\begin{equation*}
\begin{aligned}
	||G_2||_{L^\infty(I)}&\le||I^{\alpha,\beta}_N\left(S_1\bar e\right)-S_1\bar e||_{L^\infty(I)}+||I^{\alpha,\beta}_N\left(S_2\bar e\right)-S_2\bar e||_{L^\infty(I)}\\
	&\le||I^{\alpha,\beta}_N\left(S_1\bar e\right)-T_N\left(S_1 \bar e\right)||_{L^\infty(I)}+||T_N\left(S_1 \bar e\right)-S_1 \bar e||_{L^\infty(I)}\\
	&+||I^{\alpha,\beta}_N\left(S_2\bar e\right)-T_N\left(S_2 \bar e\right)||_{L^\infty(I)}+||T_N\left(S_2 \bar e\right)-S_2 \bar e||_{L^\infty(I)}\\
	&\le(1+||I_N^{\alpha,\beta}||_\infty)\left(||T_N\left(S_1 \bar e\right)-S_1 \bar e||_{L^\infty(I)}+||T_N\left(S_2 \bar e\right)-S_2 \bar e||_{L^\infty(I)}\right)\\
	&\le c (1+||I_N^{\alpha,\beta}||_\infty) N^{-\kappa} \left(||S_1\bar{e}||_{C^{0,\kappa}}+||S_2\bar{ e}||_{C^{0,\kappa}}\right)\\
	&\le c (1+||I_N^{\alpha,\beta}||_\infty) N^{-\kappa} ||\bar{e}||_{L^{\infty}(I)}.
\end{aligned}
\end{equation*}
Thus, according to Lemma \ref{interpolationerror} and the above estimate, we have
\begin{equation}\label{G2}
		||G_2||_{L^\infty(I)}\le\left\{\begin{aligned}
		&cN^{-\kappa} \ln N||\bar{e}||_{L^\infty(I)},&-1<\alpha,\beta\le-\frac{1}{2},\\
		&cN^{\gamma+\frac{1}{2}-\kappa}|\bar{e}||_{L^\infty(I)},&\textup{otherwise}.
		\end{aligned}\right.
\end{equation}
Then, under the assumption that $\kappa\ge 0$ and $\gamma+\frac{1}{2}-\kappa<0$, for sufficiently large $N$, we have
\begin{equation*}
	||G_2||_{L^\infty(I)}\le \frac{1}{3}||\bar{e}||_{L^{\infty}(I)}.
\end{equation*}
Combining \eqref{error1}, \eqref{ING}, \eqref{G1} and \eqref{G2} gives the desired estimate in \eqref{errorestimate}.
\end{proof}
\section{Asymptotic compatibility}
We now prove the asymptotic compatibility of the spectral collocation methods, which means that the spectral collocation solution converges to its correct local limit as $1/N$ and $\delta$ both go to zero. This provides natural links between nonlocal models and traditional PDE models. Here, $u_0$ is denoted as the solution of the following traditional diffusion equation:
\begin{equation}\label{localequation}
	\left\{\begin{aligned}
	&-C u_0''(x)=f_0(x),&x\in I\;\\
	&u_0(x)=g_0(x),&x\in\partial I.
	\end{aligned}\right.
\end{equation}
Then, the asymptotic relation between $u_0$ and $u_N$ is established in the following theorem.
\begin{theorem}\label{actheorem}
Assume that all the conditions in Theorem \ref{convergence} hold. Let $u,u_N$ be the solutions of the nonlocal diffusion equation in \eqref{modelequation} and the collocation scheme in \eqref{collocationscheme}, respectively. If $u\in H^m(I)\cap C^4(I)$ and $m\ge1$ is an integer, $g(x)=g_0(x)$ for $x\in \partial I$, and
\begin{equation*}
			||f(x)-f_0(x)||_{\omega^{0,0}}\le c\delta^\lambda.
	\end{equation*}
We further assume that the kernel function $\gamma_\delta$ satisfies
\begin{equation}\label{moment}
		\int_{x-\delta}^{x+\delta}\gamma_\delta(x,y)\frac{1}{|y-x|^\mu}(y-x)^2dy=2C.
	\end{equation}
Then, it holds that
\begin{equation*}
	||u_N-u_0||_{\omega^{0,0}}\le\left\{\begin{aligned}
	&O(N^{-m}\ln N)+O(N^{\frac{1}{2}-r}\ln N)+O(\delta^{\min\{2,\lambda\}}),&-1<\alpha,\beta\le-\frac{1}{2},\\
	&O(N^{\gamma+\frac{1}{2}-m})+O(N^{1+\gamma-r})++O(\delta^{\min\{2,\lambda\}}),&\textup{otherwise}.
	\end{aligned}\right.	
	\end{equation*}
\end{theorem}
\begin{proof}
For $x\in I$ and $y\in [x-\delta,x+\delta]$, by Taylor's theorem,
\begin{equation}\label{taylor}
		u(y)-u(x)=u'(x)(y-x)+\frac{1}{2}u''(x)(y-x)^2+\frac{1}{3!}u^{(3)}(x)(y-x)^3+u^{(4)}(\xi)(y-x)^4,
	\end{equation}
where $\xi\in[\min\{x,y\},\max\{x,y\}]$. Substituting \eqref{taylor} into \eqref{modelequation} and using \eqref{moment} yields
\begin{equation}\label{taylor2}
		-\int_{x-\delta}^{x+\delta}\gamma_\delta\frac{1}{|x-y|^\mu}\left(u(y)-u(x)\right)dy=-Cu''(x)+O(\delta^2).
	\end{equation}
Then, combining \eqref{localequation} and \eqref{taylor2} leads to
\begin{equation*}
	 	-C\left(u-u_0\right)''=f(x)-f_0(x)+O(\delta^2).
	 \end{equation*}
Based on the regularity of the solutions of Poisson's equation, we have
\begin{equation}\label{regularity}
	 	||u-u_0||_{2,\omega^{0,0}}\le||f(x)-f_0(x)+O(\delta^2)||_{\omega^{0,0}}= O(\delta^{\min\{2,\lambda\}}),
	 \end{equation}
where we use the assumption that $(u-u_0)|_{\partial I}=0$. The combination of the triangle inequality, Theorem \ref{convergence} and \eqref{regularity} implies
\begin{equation}\label{ac}
	 	\begin{aligned}
	 	||u_N-u_0||_{\omega^{0,0}}&\le||u_N-u||_{\omega^{0,0}}+||u-u_0||_{\omega^{0,0}}\\
	 	&\le\left\{\begin{aligned}
	 	&O(N^{-m}\ln N)+O(N^{\frac{1}{2}-r}\ln N)+O(\delta^{\min\{2,\lambda\}}),&-1<\alpha,\beta\le-\frac{1}{2},\\
	 	&O(N^{\gamma+\frac{1}{2}-m})+O(N^{1+\gamma-r})+O(\delta^{\min\{2,\lambda\}}),&\textup{otherwise}.
	 	\end{aligned}\right.
	 	\end{aligned}
	 \end{equation}
This completes the proof of the theorem.
\end{proof}
\section{Numerical experiments}
In this section, two numerical examples are used to verify the theoretical results of the previous sections. Without loss of generality, we use Legendre-collocation methods ($\alpha=\beta=0$) and Chebyshev-collocation methods ($\alpha=\beta=-\frac{1}{2}$) to solve weakly singular nonlocal diffusion equations \eqref{modelequation} with $\mu=\frac{1}{2}$ and $\delta=0.2$ in Example 1, and the numerical solutions are denoted by $u_N^l$ and $u_N^c$, respectively. The other Jacobi-collocation methods can be implemented similarly. Furthermore, the asymptotic compatibility of the Jacobi collocation methods is also verified in Example 2.\\

\noindent{\bf{Example 1.}} In the first example, we consider the weakly singular nonlocal diffusion equations $\eqref{modelequation}$ with a constant kernel $\gamma_\delta(x,y)=\frac{5}{2}\delta^{-\frac{5}{2}}$ and a Gaussian-type kernel $\gamma_\delta=\exp\left(-\frac{|y-x|^2}{\delta^2}\right)$. For the constant kernel, we choose an exact solution $u(x)=xe^x$ with the corresponding source function
\begin{equation*}
		f(x)=-\frac{5}{4}\delta^{-3}e^x\left(2\delta e^{-\delta}(1+e^{2\delta})-8\delta x+\sqrt{\delta}\sqrt{\pi}(2x-1)\left(\textup{erf}(\sqrt{\delta})+\textup{erfi}(\sqrt{\delta})\right)\right),
	\end{equation*}
where $\textup{erf}(x)$ denotes the Gaussian error function and is defined as
\begin{equation*}
		\textup{erf}(x)=\frac{2}{\sqrt\pi}\int_0^x e^{-t^2}dt,
	\end{equation*}
and $\textup{erfi}(x)=-i\textup{erf}(ix)$ denotes the imaginary error function with $i=\sqrt{-1}$.

For the Gaussian-type kernel, the exact solution $u(x)=x(1-x)$
is used with the source function
\begin{equation*}
	f(x)=\delta^{\frac{5}{2}}\left(\Gamma(\frac{5}{4})-\Gamma(\frac{5}{4},1)\right),
\end{equation*}
where $\Gamma(a)$ denotes a standard Gamma function and $\Gamma(a,x)$ denotes an upper incomplete Gamma function. The corresponding convergence results are shown in Figure \ref{convergenceplot} and Table \ref{convergencetable}. Clearly, the desired spectral accuracy is obtained, which verifies our theoretical results.\\
\begin{figure}[htbp]
\begin{minipage}[htbp]{55 mm}
		\centering
\centerline{\includegraphics[width=65 mm]{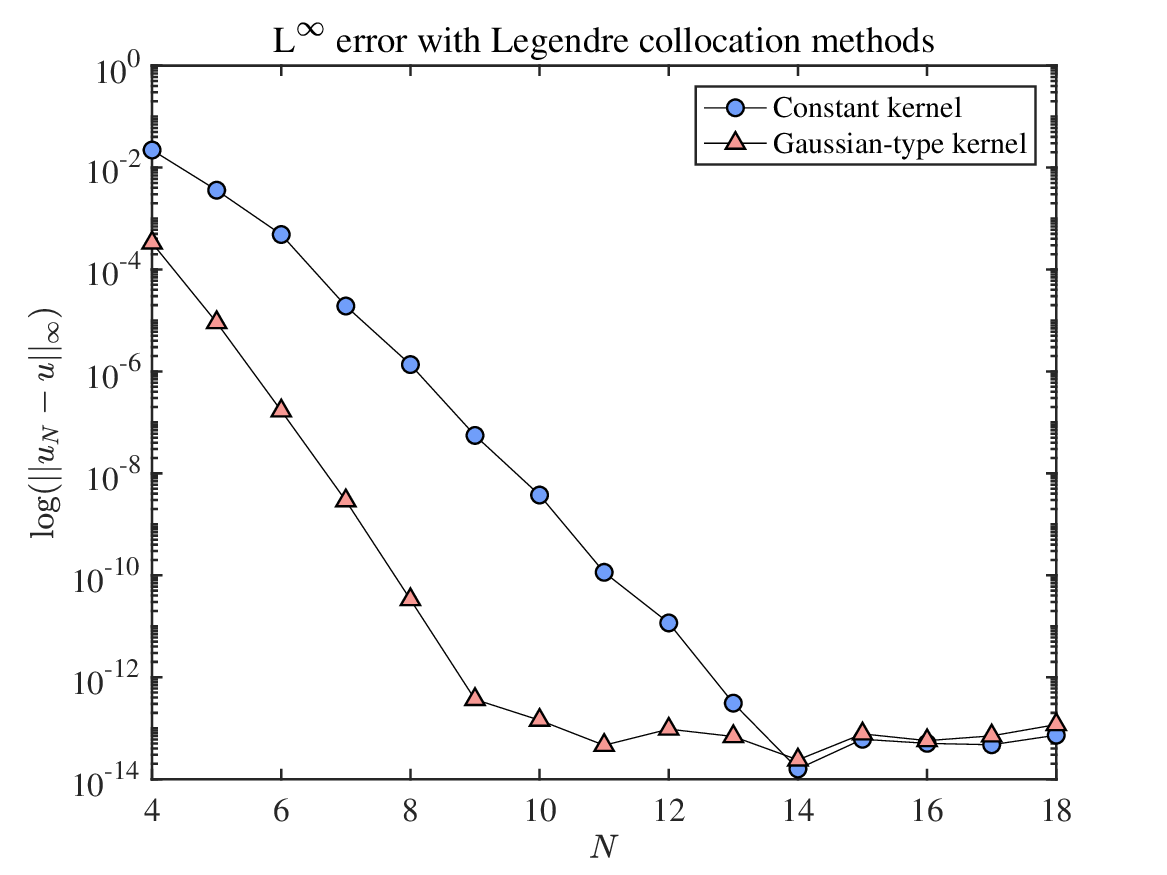}}
\centerline{(a)}
\end{minipage}
\hfil
\begin{minipage}[htbp]{55 mm}
\centerline{\includegraphics[width=65 mm]{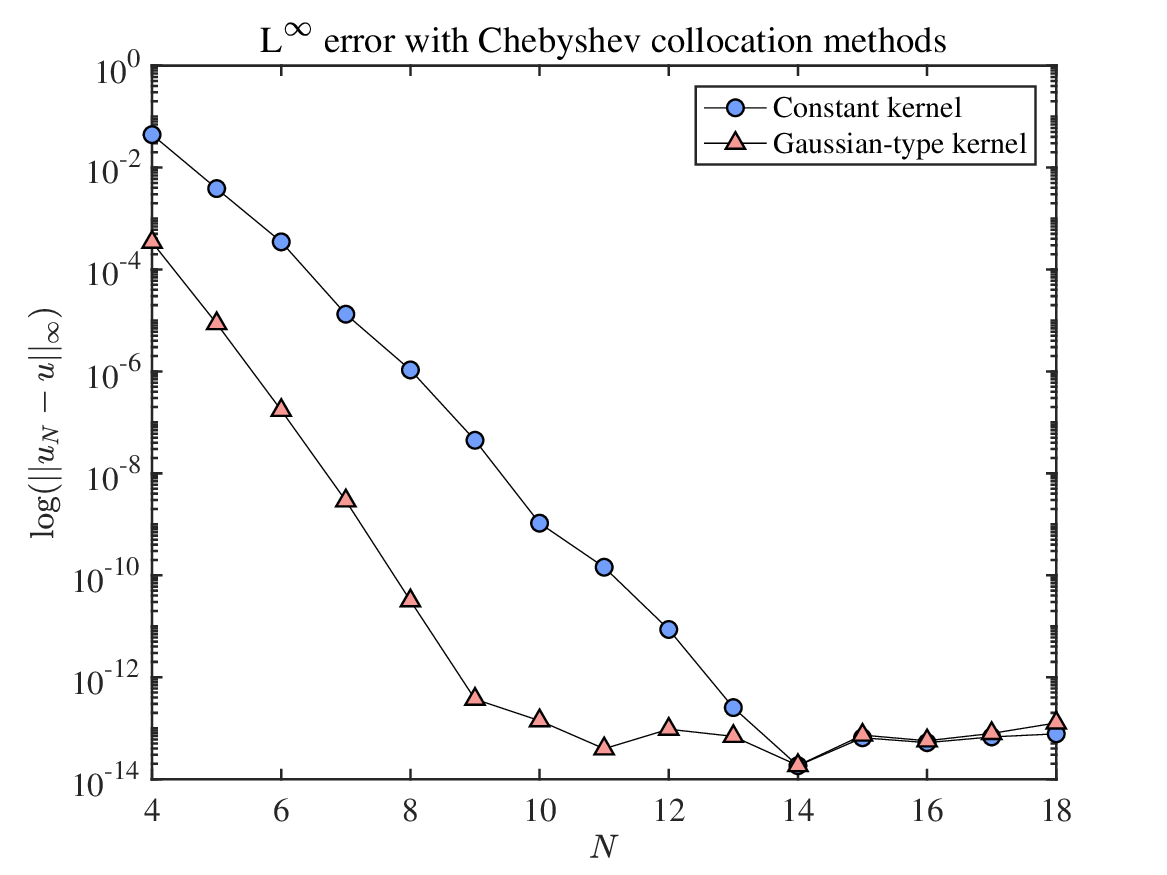}}
\centerline{(b)}
\end{minipage}
\caption{Convergence results for weakly singular nonlocal diffusion equations with two different kernel functions using Jacobi-collocation methods. (a): Legendre collocation method. (b): Chebyshev collocation method.}
	\label{convergenceplot}
\end{figure}\par

\begin{table*}[htbp]
	\scriptsize
	\centering
	\caption{Convergence results for weakly singular nonlocal diffusion equations \eqref{modelequation} with spectral collocation methods. }
	\label{convergencetable}
	\begin{tabular}{ccccccccc}
		\toprule
		\multirow{3}{*}{$N$} & \multicolumn{2}{c}{Constant kernel} & \multicolumn{2}{c}{Gaussian-type kernel}& \\
		\cmidrule(r){2-3} \cmidrule(r){4-5} 
		&  $L^\infty(u_N^l-u)$       &  $L^\infty(u_N^c-u)$   
		&  $L^\infty(u_N^l-u)$       &  $L^\infty(u_N^c-u)$    \\
		\midrule
		$4$             &2.21e-02               & 4.42e-02             &3.32e-04                & 3.43e-04           \\
		$6$             &4.85e-04              & 3.48e-04             &1.69e-07                 & 1.72e-07            \\
		$8$             &1.36e-06               & 1.07e-06              &3.34e-11                 & 3.17e-11             \\
		$10 $           &3.75e-09              & 1.05e-09              &1.43e-13                 & 1.40e-13             \\
		$12$            &1.16e-11                & 8.65e-12             &9.59e-14                 & 9.56e-14             \\
		$14$            &1.58e-14               & 1.84e-14              &2.39e-14                 & 1.85e-14             \\
		$16 $           &5.04e-14               &5.21e-14               &5.70e-14                 & 5.67e-14             \\
		$18 $           &7.29e-14               &7.75e-14                &1.17e-13                 & 1.27e-13              \\
		\bottomrule
	\end{tabular}
\end{table*}

\noindent{\bf{Example 2.}} Then, we verify the asymptotic compatibility of the spectral collocation method when the nonlocal length scale $\delta$ is coupled with $N$ by setting $\delta=\frac{1}{N}$. We consider a traditional diffusion equation $-u_0''(x) = f_0$ with an exact solution $u_0(x)=xe^x$. The source functions for the traditional and nonlocal diffusion equations are both given by $f(x)=f_0(x)=-xe^x-2e^x$. For the nonlocal diffusion equation, a constant kernel function $\gamma_\delta(x,y)=\frac{5}{2}\delta^{-\frac{5}{2}}$ is selected for simplicity; then, we have $C=\frac{1}{2}\int_{x-\delta}^{x+\delta}\gamma_\delta(x,y)\frac{1}{|y-x|^\mu}(y-x)^2dy=1$ with $\mu=\frac{1}{2}$. Additionally, the boundary conditions for both the traditional and nonlocal diffusion equations are selected to be compatible with the exact solution. The convergence profiles are shown in Figure \ref{delta_convergence}, which shows that the compatibility error is dominated by $O(\delta^2)$ and further verifies Theorem \ref{actheorem}.

\begin{figure}[htbp]
	\centering
	\includegraphics[width=75mm]{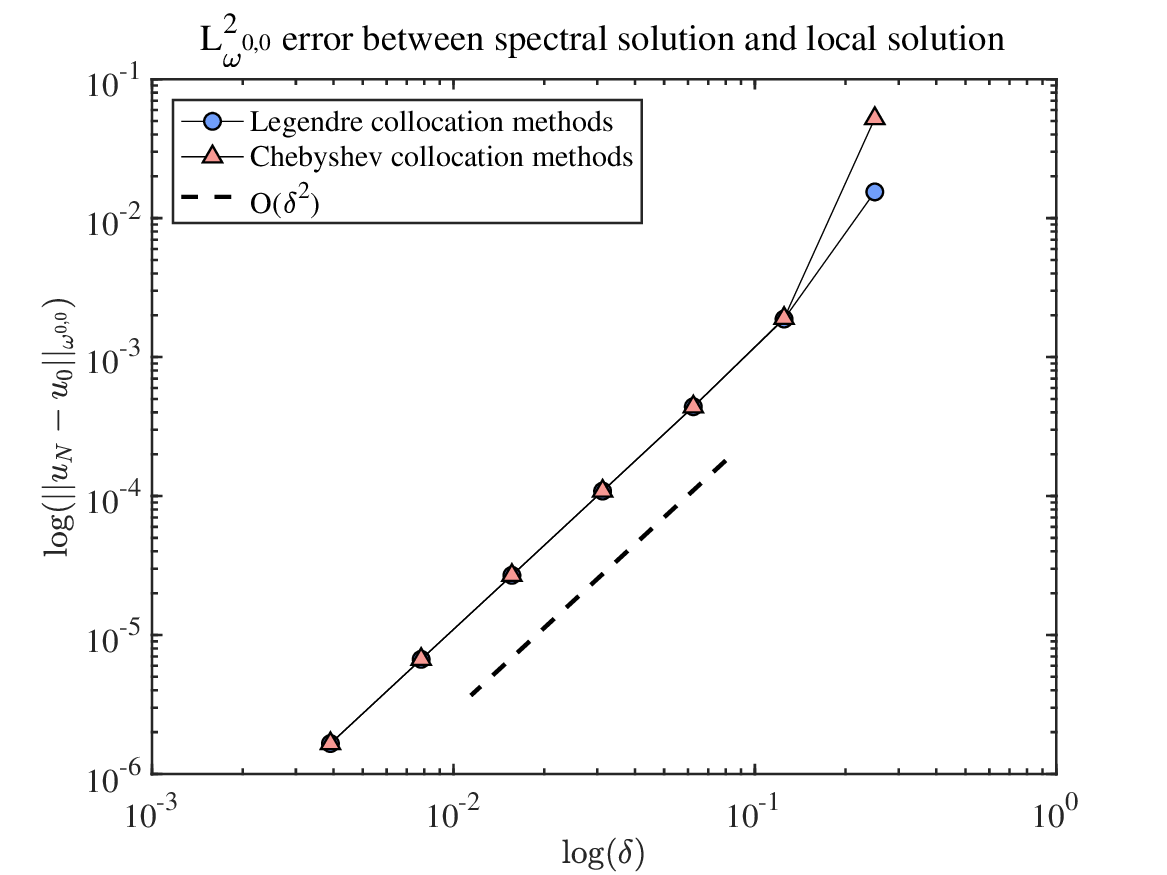}
\caption{Asymptotic compatibility verification of Jacobi collocation methods with $\delta=1/N\rightarrow0$.}
	\label{delta_convergence}
\end{figure}\par

\section{Conclusion}
In this paper, we develop a Jacobi spectral collocation method for weakly singular nonlocal diffusion equations with volume constraints. To reach high-order accuracy for the approximation of a weak singular nonlocal integral, two-sided Jacobi spectral quadrature rules are established, and the spectral rate of convergence for the proposed method is established with the $L^\infty$ norm. Moreover, we theoretically show that the discrete solution obtained by applying the Jacobi collocation method to a weakly singular nonlocal diffusion equation converges to the correct local limit as nonlocal interactions vanish. The numerical results are presented to verify the effectiveness and asymptotic compatibility of the proposed method. The possibility of extending such spectral collocation methods to nonlocal diffusion models in high dimensions will be investigated in future works.
\section*{Acknowledgments}
This research was supported by the National Natural Science Foundation of China (No. 11971386) and the National Key R\&D Program of China (No. 2020YFA0713603).
\bibliography{mybibfile}

\end{document}